\documentclass{amsart}

\usepackage[latin1]{inputenc}
\usepackage{amssymb,amsmath}
\usepackage[mathscr]{eucal}
\usepackage{enumerate}

\newtheorem{definition}{Definition}[section]
\newtheorem{theorem}[definition]{Theorem}
\newtheorem{lemma}[definition]{Lemma}
\newtheorem{remark}[definition]{Remark}
\newtheorem{corollary}[definition]{Corollary}
\newtheorem{proposition}[definition]{Proposition}

\newcommand{\M}{\mathcal{M}}
\newcommand{\MB}{\mathcal{BM}}
\newcommand{\MBR}{\mathcal{BM}_{(\Phi_1,\Phi_2,\Phi_3)}(\R)}
\newcommand{\mul}{\mathcal{\tilde M}_{(\Phi_1,\Phi_2,\Phi_3)}(\R)}

\newcommand{\C}{\mathbb{C}}

\newcommand{\N}{\mathbb{N}}
\newcommand{\R}{\mathbb R}

\newcommand{\ee}{\end{equation}}
\newcommand{\be}{\begin{equation}}
\newcommand{\ea}{\end{eqnarray*}}
\newcommand{\ba}{\begin{eqnarray*}}

\begin{document}

\def\dem{\noindent {\sf  Proof.\ \ }}
\def\qed{\hfill $\blacksquare$}

\textwidth 16cm  % Cambio tama\~{n}o m\'{a}rgenes
\textheight 21cm \headsep 1.75cm \topmargin -1.5cm \oddsidemargin
0in \evensidemargin 0in

\title{Notes on bilinear multipliers on Orlicz spaces}

\author{Oscar Blasco}

\address{Department of Mathematics,
Universitat de Valencia, Burjassot 46100 (Valencia)
 Spain}
\email{oscar.blasco@uv.es}

\author{Alen Osancliol}

\address{Department of Mathematics,
Universitat de Valencia, Burjassot 46100 (Valencia)
 Spain}
\email{alen.osancliol@ext.uv.es}

\keywords{bilinear multipliers, Orlicz spaces}
\thanks{
The first author is partially supported by  Proyecto MTM2014-53009-P(MINECO Spain) and the second author is supported by "The Scientific and Technological Research Council of Turkey" TUBITAK-BIDEB grant no 1059B191600535} \maketitle

\begin{abstract} Let $\Phi_1 , \Phi_2 $ and $ \Phi_3$ be Young functions and let $L^{\Phi_1}(\mathbb{R})$, $L^{\Phi_2}(\mathbb{R})$ and $L^{\Phi_3}(\mathbb{R})$ be the corresponding Orlicz spaces.
We say that a  function $m(\xi,\eta)$ defined on $\mathbb{R}\times \mathbb{R}$ is a bilinear multiplier of type $(\Phi_1,\Phi_2,\Phi_3)$ if
\[ B_m(f,g)(x)=\int_\mathbb{R} \int_\mathbb{R} \hat{f}(\xi) \hat{g}(\eta)m(\xi,\eta)e^{2\pi i (\xi+\eta) x}d\xi d\eta \]
defines a bounded bilinear operator from $L^{\Phi_1}(\mathbb{R}) \times L^{\Phi_2}(\mathbb{R})$ to $L^{\Phi_3}(\mathbb{R})$. We denote by $BM_{(\Phi_1,\Phi_2,\Phi_3)}(\mathbb{R})$ the space of all bilinear multipliers of type $(\Phi_1,\Phi_2,\Phi_3)$ and investigate some properties of such a class. Under some conditions on the triple $(\Phi_1,\Phi_2,\Phi_3)$ we give some examples of bilinear multipliers of type $(\Phi_1,\Phi_2,\Phi_3)$. We will focus on the case $m(\xi,\eta)=M(\xi-\eta) $ and get necessary conditions on $(\Phi_1,\Phi_2,\Phi_3)$ to  get non-trivial multipliers in this class. In particular  we recover some of the the known results for Lebesgue spaces.
\end{abstract}

%%%%%%%%%%%%%%%%%%%%%%%%%%%%%%%%%%%%%%%%%%%%%%%%%%%%%%%%%%%%%%

\

\section{Introduction.}

The theory of bilinear multipliers was originated
in the work by R. Coiffman and C. Meyer (\cite{CM}) in the eighties of the last century and continued by L. Grafakos and R. Torres (\cite{GT}) and many others. A renewed interest appeared in the
nineties after the celebrated result by M. Lacey and C. Thiele (\cite{LT1, LT2}), solving the old standing conjecture of Calder\'on on the boundedness of the bilinear
Hilbert transform. Let us recall that for a couple of functions  $f, g:\R\to \C$ such that $\hat f$ and $\hat g$ are compactly supported and  for any locally integrable function $m(\xi,\eta)$ defined on $\mathbb{R}\times \mathbb{R}$ one can consider the mapping
\be\label{bm} B_m(f,g)(x)=\int_\mathbb{R} \int_\mathbb{R} \hat{f}(\xi) \hat{g}(\eta)m(\xi,\eta)e^{2\pi i (\xi+\eta) x}d\xi d\eta \ee
and ask himself about its boundedness on certain function spaces.
In such a way the bilinear versions of  several classical operators appearing
in Harmonic Analysis, such as the Hilbert transform, the fractional integrals, the Hardy-Littlewood maximal function and many others   have been considered in
the last decades and their boundedness on several spaces have been addressed.

The study of bilinear multipliers for smooth symbols (where
$m(\xi,\eta)$  is a ``nice" regular function with at most a single point singularity) goes back to the
work by R.R. Coifman and Y. Meyer in \cite{CM}.  A particularly interesting case is
$m(\xi,\eta)=M(\xi-\eta)$ for a measurable function $M$ where, for instance the case
$M(x)=\frac{1}{|x|^{1-\alpha}}$, corresponding to the bilinear fractional transform,  was shown to define a bilinear
multiplier mapping $L^{p_1}(\R)\times L^{p_2}(\R)$ into $L^{p_3}(\R)$ for $1/p_3=1/p_1+1/p_2-\alpha$
for $1<p_1,p_2<\infty$ and $0<\alpha<1/p_1+1/p_2$ (see \cite{KS,GK}) or the celebrated result of the bilinear Hilbert transform, given by the case $M(x)=sign (x)$, was shown to define a
bilinear multiplier of type $(p_1,p_2, p_3)$ for
$1/p_3=1/p_1+1/p_2$ for $1<p_1,p_2<\infty$ and $p_3>2/3$ (\cite{LT1, LT2,LT3}). The case of more general non-smooth symbols was later analyzed by J. Gilbert and A. Namod (see \cite{GN1,GN2}).

Bilinear multipliers acting on  other groups such as torus $\mathbb{T}$ or integers $\mathbb{Z}$ have also been studied. Their corresponding analogues have been achieved using transference properties first by D. Fan and S. Sato \cite{FS} and later by the results in several papers by E. Berkson, O.Blasco, M.J. Carro and A.Gillespie (see \cite{B, BCG,BBCG, BBCG2}). More recently several results on bilinear multipliers defined on  locally compact abelian groups and acting on rearrangement invariant quasi-Banach spaces have been obtained by S. Rodriguez-L\'opez \cite{R}.
Other function spaces such as Lorentz spaces  have been studied mainly by O. Blasco and F. Villarroya (see \cite{BV, V}) and for also for weighted Lebesgue spaces or Lebesgue spaces with variable exponent  by T. G\"urkanli and O. Kulak \cite{GK}. Our objective will be to deal with bilinear multipliers on $\mathbb R$ (although similar results can be presented in $\R^n$) acting on Orlicz spaces.

 Throughout the paper ${\mathcal P}(\R)$ stands for the
set of functions  such that $supp \hat f$
is compact and $S(\R)$ for the Schwartz class on $\R$, i.e. $f:\R \to \C$ such that $f\in C^\infty(\R)$ and $x^kf^{(n)}(x)$ is bounded for any $k$ and $n$.
 We write the Fourier transform  by $\hat f(\xi)= \int_{\R} f(x) e^{-2\pi
i x\xi}dx$  and  we denote the translation by $\tau_yf(x) = f(x-y)$, the modulation by $M_xf(y) = f(y)e^{2\pi\textit{i}xy}$ and the dilation by $D_\lambda f(x)=f(\lambda x)$ for $x,y \in \R$ and $\lambda>0$.  As usual for $g$ defined in $\R^n$ we write $g_t(x)=\frac{1}{t^n}g(\frac{x}{t})=\frac{1}{t^n}D_{1/t}g(x)$ for $x\in \R^n$ and $t>0$. Clearly one has for each $f\in L^1(\R)$, $y\in \R$ and $\lambda>0$
$$
\widehat{(\tau_yf)}(\xi) = M_{-y}\hat{f}(\xi), \quad
 \widehat{(M_xf)}(\xi)  = \tau_x \hat{f}(\xi), \quad
 \widehat{(D_\lambda f)}(\xi)  = \hat{f}_\lambda(\xi).
$$

Given a Young function $\Phi$, the Orlicz space $L^\Phi(\R)$ consists of the set of all measurable functions $f:\R \to \C$ such that $\int_\R \Phi(\vert f(x)\vert/\lambda)dx < \infty$ for some $\lambda>0$, which equipped with the so called Luxemburg norm
\[ N_\Phi(f) = \inf \{ \lambda>0 : \int_\R \Phi(\vert f(x)\vert/\lambda)dx \leq 1 \} \]
becomes a Banach space.

It is known that if a Young function $\Phi$ satisfies the $\Delta_2$-condition (i.e. there exists a constant $k>0$ such that $\Phi(2x) \leq k\Phi(x)$ for all $x\geq 0$), then the space of compactly supported functions in $C^\infty(\R)$ is dense in $L^\Phi(\R)$ with respect to the norm $N_\Phi(\cdot)$. Hence, in this case $S(\R)$ and $\mathcal P(\R)$ are also dense in $L^\Phi(\R)$.

Given two Young functions $\Phi_1$ and $\Phi_2$ the space
   $\M_{\Phi_1,\Phi_2}(\R)$ stands
for the space of bounded functions $m$ defined on $\R$ such that
\begin{equation}\label{lm}T_m (f)(x)= \int_\R \hat
f(\xi) m(\xi) e^{2\pi i\xi x} d\xi
\end{equation}
defines a bounded operator from $L^{\Phi_1}(\R)$
to $L^{\Phi_2}(\R)$. We endow the space with the ``norm" of the
operator $T_m$, that is $\|m\|_{\Phi_1,\Phi_2}=\|T_m\|$.
We refer the reader to \cite{BL, SW} for  the case $\Phi_1(x)=|x|^p$ and $\Phi_2(x)=|x|^q$, to be denoted $\M_{p,q}(\R)$.

\begin{definition} Given  three Young functions  $\Phi_i$ for $i=1,2,3$, a locally integrable function   $m$ defined on $\R^2$ is said to be a {\it bilinear
multiplier}  of type $(\Phi_1,\Phi_2; \Phi_3)$  if there exists a constant $C>0$ such that $$ B_m(f,g)(x)=\int_{\R^n} \int_{\R^n}\hat f(\xi)\hat g(\eta)
 m(\xi,\eta)e^{2\pi i \langle \xi+\eta,x\rangle}d\xi d\eta$$ satisfies  $$N_{\Phi_3}(B_m(f,g))\le
CN_{\Phi_1}(f)N_{\Phi_2}(g)$$ for any $f, g\in \mathcal P(\R)$.

We write $\MB_{(\Phi_1,\Phi_2; \Phi_3)}(\R)$ for the space of bilinear
multipliers of type  $(\Phi_1,\Phi_2; \Phi_3)$  and
$\|m\|_{(\Phi_1,\Phi_2; \Phi_3)}= \|B_m\|$.

  We
 denote by $\tilde\M_{(\Phi_1,\Phi_2,\Phi_3)}(\R)$ the space of
  locally integrable  functions $M$ defined on $\R$ such that
 $m(\xi,\eta)=M(\xi-\eta)\in \MB_{(\Phi_1,\Phi_2; \Phi_3)}(\R)$.
\end{definition}

Note that in the case that $\Phi_1$ and $\Phi_2$ satisfy $\Delta_2$-condition then $M\in \MB_{(\Phi_1,\Phi_2; \Phi_3)}(\R)$ means that
 $$ B_M(f,g)(x)=\int_{\R^n} \int_{\R^n}\hat f(\xi)\hat g(\eta)
 M(\xi-\eta)e^{2\pi i \langle \xi+\eta,x\rangle}d\xi d\eta$$
 extends to a bounded bilinear map from
 $L^{\Phi_1}(\R)\times L^{\Phi_2}(\R)$ into $L^{\Phi_3}(\R)$.
 We keep the notation $\|M\|_{(\Phi_1,\Phi_2; \Phi_3)}= \|B_M\|.$
This generalize the case $\Phi_i(x)=x^{p_i}$ considered in \cite{B1} and denoted $\MB_{(p_1,p_2,p_3)}(\R)$ and
$\tilde\M_{(p_1,p_2,p_3)}(\R)$ respectively.

In this paper, we shall  investigate some  properties of the spaces $\MB_{(\Phi_1,\Phi_2, \Phi_3)}(\R)$ and $\tilde\M_{(\Phi_1,\Phi_2, \Phi_3)}(\R)$.
The paper is divided into five sections. The first section is devoted to recall some notions on Orlicz spaces to be used in the sequel. In particular we shall analyze the norm of the dilation operator $D_\lambda$ acting on Orlicz spaces. In Section 3 we shall give elementary examples of bilinear multipliers and procedures to generate them. In Section 4 we mainly  focus on the case $m(\xi,\eta)=M(\xi-\eta)$ and give some sufficient conditions to define a bilinear multiplier on Orlicz spaces. Finally we use the last section to investigate some necessary conditions to get a non-zero bilinear multipliers in the class $\tilde\M_{(\Phi_1,\Phi_2; \Phi_3)}(\R)$, generalizing  the known results for Lebesgue spaces.

\section{Orlicz spaces}

A non-zero function $\Phi :\R \to [0,\infty]$ is called a Young function if $\Phi$ is convex, even and $\Phi(0)=0$. If $\Phi$ is a Young function then $\Phi^{-1}$ is defined for $0\leq y$ by
$$\Phi^{-1}(y)=\inf\{x>0:\Phi(x)>y\}$$ where $\inf \emptyset=\infty$ and it is easy to see \cite{ON} that
\be \label{new}\Phi(\Phi^{-1}(x))\le x\le \Phi^{-1}(\Phi(x)),\quad  x\ge 0.
\ee

%If $F:\R^+\to \R^+$ is non-decreasing and left continuous we consider the generalized inverse of $F$
%$$F^{-1}(\lambda)=\inf\{t>0:F(t)>\lambda\}.$$
%It is easy to see that $F^{-1}$ is non-decreasing and right continuous. Moreover
%\be \label{new}F(F^{-1}(\lambda))\le \lambda\le F^{-1}(F(\lambda)).
%\ee

Given a Young function $\Phi$, its complementary function  is defined by
\[ \Psi(y)=\sup\{ x|y| - \Phi(x)  : x\geq0 \} \]
for $y\in \mathbb{R}$. It can be seen that $\Psi$ is still a Young function in the sense of above definition. Then $(\Phi,\Psi)$ is called a complementary pair of Young functions and they satisfy
\be\label{comp}   |x| \leq \Phi^{-1}(x)\Psi^{-1}(x)\leq 2|x|, x\in \R, \ee
and the Young inequality
\be\label{yi} \vert xy \vert \leq \Phi(x)+\Psi(y) ,x,y \in \R. \ee

There are several inequalities to be used throughout the paper when dealing with Orlicz spaces: One deals with the generalization of H\"older's inequality (see \cite{ON},\cite[page 64]{RR}): Let $\Phi_i$, $i=1,2,3$ be Young's functions satisfying
\be \label{hi} \Phi_1^{-1}(x) \Phi_2^{-1}(x) \leq \Phi_3^{-1}(x), \quad x\geq 0.\ee
If $f \in L^{\Phi_1}(\R)$ and $g \in L^{\Phi_2}(\R)$ then $f g \in L^{\Phi_3}(\R)$ and
\begin{equation} \label{holder}
 N_{\Phi_3}(fg) \leq 2 N_{\Phi_1}(f) N_{\Phi_2}(g).
\end{equation}

The other one refers to Young's inequality for convolutions (see \cite{ON},\cite[page 64]{RR}): Let  $\Phi_i$, $i=1,2,3$ be Young functions satisfying \be\label{yi}\Phi_1^{-1}(x) \Phi_2^{-1}(x) \leq x\Phi_3^{-1}(x), \quad x\geq 0. \ee If $f \in L^{\Phi_1}(\R)$ and $g \in L^{\Phi_2}(\R)$ then the convolution $f\star g\in L^{\Phi_3}(\R)$ and
\begin{equation} \label{young}
 N_{\Phi_3}(f\star g) \leq 2 N_{\Phi_1}(f) N_{\Phi_2}(g).
\end{equation}
The reader is referred to \cite{RR} for the proofs of these results and for further information about Orlicz spaces.

In this section, we shall give some estimates to the norms of the dilation operator on Orlicz spaces which will be useful in the sequel.

Given $\gamma>0$ one can define
\[ N_{\Phi,\gamma}(f)=\inf\{ k>0 : \int_\R \Phi(\frac{\vert f(x)\vert}{k}) dx \leq \gamma \} .\]
Of course $N_{\Phi,1}=N_{\Phi}$. Let us observe that these quantities give equivalent norms in $L^\Phi(\R)$. In fact, by convexity, we can easily see the following property of these norms: If $ 0<\gamma_1<\gamma_2$ and $f$ is a measurable function then
\be \label{in1} \frac{\gamma_1}{\gamma_2}N_{\Phi,\gamma_1}(f)\le N_{\Phi,\gamma_2}(f) \leq N_{\Phi,\gamma_1}(f). \ee

%Also, note that for the Fourier transform of the function $G_\lambda$ we have

%\[ \hat{G_\lambda}(\xi) = \hat{G}(\xi \lambda). \]

%Consider the dilation operator $D_\lambda :L^{\Phi_i}(\R) \to L^{\Phi_i}(\R)$,

%\begin{equation}
% D_\lambda G(x)= G_\lambda(x) = \frac{1}{\lambda}G(\frac{x}{\lambda})
%\end{equation}
Throughout the paper $$C_\Phi(\lambda)=\|D_\lambda\|_{L^\Phi(\R)\to L^\Phi(\R)}.$$
Of course $C_\Phi(\lambda)$ is non-increasing, submultiplicative and $C_\Phi(1)=1$.

\begin{proposition} \label{p0} Let $\lambda>0$ and $\Phi$ a Young function. Then
$$ \frac{1}{\max\{1,\lambda\}}\le C_\Phi(\lambda)\le \frac{1}{\min\{1,\lambda\}}.$$
\end{proposition}
\begin{proof}
It is straightforward that for $f\in L^\Phi(\R)$ and $\lambda>0$ one has
\be \label{dil0}
 N_\Phi(D_\lambda f)= N_{\phi, \lambda} (f).
\ee

Using now  (\ref{in1})  we have
$$ N_{\Phi,\lambda}(f)\le N_\Phi(f)\le \lambda N_{\Phi,\lambda}(f), \quad \lambda\ge 1$$
and
$$ \lambda N_{\Phi,\lambda}(f)\le N_\Phi(f)\le  N_{\Phi,\lambda}(f), \quad 0<\lambda\le 1$$
The result now follows from (\ref{dil0}).
\end{proof}

Let us now get better estimates for $C_\Phi(\lambda)$ using the following lemma.

\begin{lemma} \label{l1} Let $\Phi$ be a Young function and $A\subset \R$ be measurable with $0<|A|<\infty$.
If $f$ be a bounded function supported on $A$ then
$$ \frac{\|f\|_1}{|A|\Phi^{-1}(|A|^{-1})}\le N_\Phi(f)\le\frac{\|f\|_\infty}{\Phi^{-1}(|A|^{-1})}$$
where $|A|$ stands for the Lebesgue measure of $A$.

In particular if $|f(x)|=1$ for $x\in A$ then $N_\Phi(f)=\frac{1}{\Phi^{-1}(|A|^{-1})}$.
\end{lemma}
\begin{proof} From (\ref{new}) one sees that $\{x>0: \Phi(x) \leq a \}=\{x>0: x\leq \Phi^{-1}(a)\}$ for $a>0$. Therefore since $ |f(x)|\le \|f\|_\infty\chi_A(x)$ we have
\ba
N_\Phi(f) &=&  \inf \{ k>0 : \int_A \Phi(\frac{\vert f(x)\vert )}{k})dx \leq 1 \}\\
&\le&  \inf \{ k>0 : \int_A \Phi(\frac{\|f\|_\infty }{k})dx \leq 1 \}\\
&= & \inf \{ k>0 : \Phi(\frac{\|f\|_\infty }{k}) \leq |A|^{-1} \}\\
&=&  \inf \{ k>0 : \frac{\|f\|_\infty}{ \Phi^{-1}(|A|^{-1}) } \leq k\}\\
&=& \frac{\|f\|_\infty}{\Phi^{-1}(|A|^{-1})}.
\ea

For the other inequality we use Jensen inequality for convex functions.  Indeed
\ba
N_\Phi(f)&=&  \inf \{ k>0 : \frac{1}{|A|}\int_A \Phi(\frac{\vert f(x)\vert }{k})dx \leq \frac{1}{|A|} \}\\
&\ge&  \inf \{ k>0 : \Phi(\frac{1}{|A|}\int_A\frac{\vert f(x)\vert }{k})dx) \leq |A|^{-1} \}\\
&=&  \inf \{ k>0 : \frac{\|f\|_1}{|A|k}\leq \Phi^{-1}(|A|^{-1}) \}\\
&=& \frac{\|f\|_1}{|A|\Phi^{-1}(|A|^{-1})}.
\ea
\end{proof}

\begin{proposition}  \label{p1}Let $\Phi$ be a Young function. Then  $C_\Phi(\lambda)\ge \sup_{\mu>0}\frac{\Phi^{-1}(\mu)}{\Phi^{-1}(\lambda\mu)}.$
\end{proposition}
\begin{proof} Taking $A=[0,a]$ and $f=\chi_A$ in Lemma \ref{l1}, since $D_\lambda (\chi_{[0,a]})= \chi_{[0, \frac{a}{\lambda}]}$ one obtains
$$ N_\Phi( D_\lambda \chi_{[0,a]})= N_\Phi(\chi_{[0, \frac{a}{\lambda}]})=\frac{1}{\Phi^{-1}(\frac{\lambda}{a})}, \quad N_\Phi( f)=\frac{1}{\Phi^{-1}(\frac{1}{a})}.$$
Hence
$$C_\Phi(\lambda)\ge \sup_{a>0} \frac{N_\Phi( D_\lambda \chi_{[0,a]})}{N_\Phi( \chi_{[0,a]})}= \sup_{\mu>0}\frac{\Phi^{-1}(\mu)}{\Phi^{-1}(\lambda\mu)}.$$
\end{proof}

\begin{theorem} \label{submul}
Let $\Phi$ be a Young function.

(i) If $\Phi(st) \geq \Phi_1(s)\Phi(t)$ for all $s,t \ge 0$ for some $\Phi_1:\R^+\to \R^+$ non-decreasing and left continuous then $ C_\Phi(\lambda) \leq \Phi_1^{-1}(\frac{1}{\lambda}).$

(ii)If $\Phi(st) \leq \Phi_2(s)\Phi(t)$ for all $s,t \ge 0$ for some $\Phi_2:\R^+\to \R^+$ non-decreasing and left continuous then
$ C_\Phi(\lambda) \leq \frac{1}{\Phi_2^{-1}(\lambda)}. $

\end{theorem}
\begin{proof} (i) Assume that $\Phi(st) \geq \Phi_1(s)\Phi(t)$ for $s,t\ge 0$.
Note that for any $s>0$ and $k>0$ we have
$$\Phi_1(s)\int_\R\Phi(\frac{|D_\lambda f(x)|}{k}) dx=\frac{\Phi_1(s)}{\lambda}\int_\R\Phi(\frac{|f(x)|}{k}) dx\le \frac{1}{\lambda} \int_\R\Phi(s\frac{|f(x)|}{k}) dx.$$
In particular whenever  $\Phi_1(s)>\frac{1}{\lambda}$ one obtains that
$$\int_\R\Phi(\frac{|D_\lambda f(x)|}{k}) dx\le \int_\R\Phi(s\frac{|f(x)|}{k}) dx.$$
Select a decreasing sequence $s_n$ converging to $\Phi_1^{-1}(\frac{1}{\lambda})$ and invoke the Lebesgue convergence theorem to get
$$\int_\R\Phi(\frac{|D_\lambda f(x)|}{k}) dx\le  \int_\R\Phi(\Phi_1^{-1}(\frac{1}{\lambda})\frac{|f(x)|}{k}) dx.$$
Therefore for $k_\lambda=\Phi_1^{-1}(\frac{1}{\lambda})N_\Phi(f) $ one gets $\int_\R\Phi(\frac{|D_\lambda f(x)|}{k_\lambda}) dx \le 1.$
This gives that $N_\Phi(D_\lambda f)\le \Phi_1^{-1}(\frac{1}{\lambda}) N_\Phi(f)$ and we obtain (i).

(ii) Assume now $\Phi(st) \leq \Phi_2(s)\Phi(t)$. As above for $s>0$
$$\int_\R\Phi(s\frac{|D_\lambda f(x)|}{k}) dx=\frac{1}{\lambda}\int_\R\Phi(s\frac{|f(x)|}{k}) dx\le \frac{\Phi_2(s)}{\lambda} \int_\R\Phi(\frac{|f(x)|}{k}) dx.$$

Choosing $s=\Phi_2^{-1}(\lambda)$  one obtains from (\ref{new}) that $\Phi_2(s)\le \lambda$. Hence
$$\int_\R\Phi(\frac{\Phi_2^{-1}(\lambda)|D_\lambda f(x)|}{k}) dx\le   \int_\R\Phi(\frac{|f(x)|}{k}) dx.$$
Now selecting  $k=N_\Phi(f)$ we get $N_\Phi(D_\lambda f)\le \frac{1}{\Phi_2^{-1}(\lambda)} N_\Phi(f)$.  This finishes the proof of (ii).
\end{proof}

Invoking Theorem \ref{submul} and Proposition \ref{p1}  we obtain the following result.
\begin{corollary} Let $\Phi$ be a Young function satisfying $\Phi(st) \leq \Phi(s)\Phi(t)$ for all $s,t \ge 0$.  Then
\[ \frac{\Phi^{-1}(1)}{\Phi^{-1}(\lambda)}\le C_\Phi(\lambda) \leq \frac{1}{\Phi^{-1}(\lambda)}. \]
\end{corollary}

\begin{remark} %(1) Using $\Phi_1(t)= \min\{1,t\}$ or $\Phi_2(t)= \max\{1,t\}$ in Proposition \ref{submul} allows us  to recover the lower estimate in Proposition \ref{p0}.

%(2)
If $\Phi$ is sub-multiplicative and $\Phi(1)=1$ then $C_\Phi(\lambda)=\frac{1}{\Phi^{-1}(\lambda)}$. This is the case for $\Phi(x)= |x|^p$ where we obtain $C_\Phi(\lambda)=\lambda^{-1/p}$.
\end{remark}

\section{Bilinear multipliers: The basics}

Let us start with some elementary properties  of the bilinear
multipliers acting on Orlicz spaces. We follow the arguments in \cite{B1} where the case of Lebesgue spaces was studied.
Since the norm in Orlicz spaces is invariant under translations and modulations one can easily obtain the following results.

\begin{proposition}  \label{l} Let $\Phi_i$ for $i=1,2,3$ and $\tilde\Phi_j$ for $j=1,2$ be Young functions and let $m\in \MB_{(\Phi_1,\Phi_2,\Phi_3)}(\R)$.

\begin{enumerate}[\sf \ \ \ (a)]
\item If $m_1\in \M_{\tilde\Phi_1,\Phi_1}(\R)$, $m_2\in \M_{\tilde\Phi_2,\Phi_2}(\R)$
and  $\tilde m(\xi,\eta)=m_1(\xi)m(\xi,\eta)m_2(\eta)$ then $\tilde m\in \MB_{(\tilde\Phi_1,\tilde\Phi_2,\Phi_3)}(\R)$.
Moreover $$ \|\tilde m\|_{\tilde\Phi_1,\tilde\Phi_2,\Phi_3}\le \|m_1 \|_{\tilde\Phi_1,\Phi_1}\| m
\|_{(\Phi_1,\Phi_2,\Phi_3)}\| m_2\|_{\tilde\Phi_2,\Phi_2}$$

\item  If $\tau_{(\xi_0,\eta_0)} m(\xi,\eta)= m(\xi-\xi_0,\eta-\eta_0)$ then $\tau_{(\xi_0,\eta_0)} m\in \MBR$ for each $(\xi_0,\eta_0)\in
\R^2$ and $$ \|\tau_{(\xi_0,\eta_0)} m\|_{(\Phi_1,\Phi_2,\Phi_3)}= \|
m\|_{(\Phi_1,\Phi_2,\Phi_3)}.$$

\item If $M_{(\xi_0,\eta_0)} m(\xi,\eta)= e^{2\pi i(\xi_0 \xi+\eta_0\eta)}$ then $M_{(\xi_0,\eta_0)} m\in \MBR$ for each $(\xi_0,\eta_0)\in
\R^2$ and $$ \|M_{(\xi_0,\eta_0)} m\|_{(\Phi_1,\Phi_2,\Phi_3)}= \|
m\|_{(\Phi_1,\Phi_2,\Phi_3)}$$

\end{enumerate}
\label{1}
\end{proposition}

\begin{proof}
For each $f,g\in \mathcal P(\R)$ the following formulae are straightforward   \be
B_{m_1 m m_2}(f,g)= B_m(T_{m_1}f, T_{m_2}g).\ee \be
B_{\tau_{(\xi_0,\eta_0)} m}(f,g)= M_{\xi_0+\eta_0}B_m(
M_{-\xi_0}f,M_{-\eta_0}g).\ee \be B_{M_{(\xi_0,\eta_0)} m}(f,g)=
B_m( \tau_{-\xi_0}f,\tau_{-\eta_0}g).\ee
The result now follows easily.
\end{proof}

\begin{proposition} \label{l2} Let $\Phi_i$ for $i=1,2,3$ be Young functions. If $m\in \MB_{(\Phi_1,\Phi_2,\Phi_3)}(\R)$ and $t>0$ then
   $D_tm\in \MBR$. Moreover and $$\|D_tm \|_{(\Phi_1,\Phi_2,\Phi_3)}\le C_{\Phi_3}({1/t}) C_{\Phi_1}(t)C_{\Phi_2}(t)\| m\|_{(\Phi_1,\Phi_2,\Phi_3)}.$$

\end{proposition}

\begin{proof} We first observe that
\be \label{dilation}
B_{D_tm}( f,g)= D_{1/t}B_m(D_t f,D_tg)
\ee
for each $f,g\in \mathcal P(\R)$.
 Indeed, \ba
 B_m(D_t f,D_tg)(x)&=& \int_{\R^2}
 \frac{1}{t}\hat f(\frac{\xi}{t})\frac{1}{t}\hat g(\frac{\eta}{t})m(\xi, \eta) e^{2\pi
 i(\xi+\eta)x}d\xi d\eta\\
 &=& \int_{\R^2}\hat
 f(\xi)\hat g(\eta)m(t\xi, t\eta) e^{2\pi
 i(\xi+\eta)tx}d\xi d\eta\\
 &=& D_tB_{D_t m}(f,g)(x).
 \ea
  This gives
 \ba N_{\Phi_3} (B_{D_t m}(f,g))&\le & C_{\Phi_3}(1/t)N_{\Phi_3}(B_m(D_t f,D_tg))\\
 &\le &  C_{\Phi_3}({1/t})\| m\|_{(\Phi_1,\Phi_2,\Phi_3)} N_{\Phi_1}(D_t f) N_{\Phi_2}(D_tg)\\
&\le &  C_{\Phi_3}({1/t})\| m\|_{(\Phi_1,\Phi_2,\Phi_3)} C_{\Phi_1}(t)C_{\Phi_2}(t)N_{\Phi_1}( f) N_{\Phi_2}(g).
 \ea
 This completes the proof.
\end{proof}
Let us combine the previous results to  get new bilinear
multipliers from a given one.
\begin{proposition} \label{convo} Let $\Phi_i$ for $i=1,2,3$ be Young functions, $\varphi\in L^1(\R^2)$ and $m\in \MBR$. Then

\begin{enumerate}[\sf \ \ \ (a)]
\item $\varphi*m\in \MBR$ and
 $\|\varphi*m\|_{(\Phi_1,\Phi_2,\Phi_3)}\le \|\varphi\|_1 \|m\|_{(\Phi_1,\Phi_2,\Phi_3)}.$

\item $\hat\varphi m\in \MBR$ and
 $\|\hat\varphi m\|_{(\Phi_1,\Phi_2,\Phi_3)}\le \|\varphi\|_1
\|m\|_{(\Phi_1,\Phi_2,\Phi_3)}.$

\item Let $W(t)= C_{\Phi_3}({1/t}) C_{\Phi_1}(t)C_{\Phi_2}(t)$ and  $\psi\in L^1(\R^+, W)$ and assume that $t\to m(t\xi,t\eta)\psi(t)$ is integrable in $\R^+$ for each $(\xi,\eta)\in \R^2$. Define $m_\psi(\xi,\eta)=\int_0^\infty m(t\xi,t\eta)\psi(t)dt.$ Then  $m_\psi\in \MBR$ and
 $\|m_\psi\|_{(\Phi_1,\Phi_2,\Phi_3)}\le \|\psi\|_{L^1(\R^+, W)}
\|m\|_{(\Phi_1,\Phi_2,\Phi_3)}.$
\end{enumerate}
\end{proposition}

\begin{proof}
(a) Note that \ba B_{\varphi* m}(f,g)(x)&=&\int_{\R^2} \hat f(\xi)
\hat g(\eta) (\int_{\R^2} m(\xi-u, \eta-v) \varphi(u,v)dudv)e^{2\pi
i(\xi+\eta)x}d\xi d\eta\\
&=&\int_{\R^2} \big(\int_{\R^2}\hat f(\xi) \hat g(\eta)
 m(\xi-u, \eta-v) e^{2\pi
i(\xi+\eta)x}d\xi d\eta\big)\varphi(u,v)dudv\\
&=&\int_{\R^2} B_{\tau_{(u,v)}m}(f, g)(x)\varphi(u,v)dudv. \ea

 From the vector-valued Minkowski
inequality  and Proposition \ref{1} part (b), we have \ba
N_{\Phi_3}(B_{\varphi* m}(f,g)) &\le &\int_{\R^2}
N_{\Phi_3}(B_{\tau_{(u,v)}m}(f, g)) |\varphi(u,v)|dudv\\
&\le & \|m\|_{(\Phi_1,\Phi_2,\Phi_3)} N_{\Phi_1}(f) N_{\Phi_2}(g) \|\varphi\|_1.\ea

(b) Observe that\ba B_{\hat \varphi  m}(f,g)(x)&=&\int_{\R^2} \hat
f(\xi) \hat g(\eta) (\int_{\R^2} M_{(-u, -v)}m(\xi, \eta)
\varphi(u,v)dudv)e^{2\pi
i(\xi+\eta)x}d\xi d\eta\\
&=&\int_{\R^2}  B_{M_{(-u, -v)}m}(f,g)(x)\varphi(u,v)dudv.\\
\ea

Argue as above, using now Proposition \ref{1} part (c), to
conclude the result.

(c) Use the formula
  \ba
B_{m_\psi}(f,g)(x)&=&\int_{\R^2} \hat f(\xi) \hat g(\eta)
(\int_{0}^\infty D_tm(\xi, \eta) \psi(t)
dt)e^{2\pi
i(\xi+\eta)x}d\xi d\eta\\
&=&\int_{0}^\infty  B_{D_tm}(f,g)(x)
\psi(t) dt\\
\ea
and Proposition \ref{l2} to finish the proof.
\end{proof}

Let us now present an elementary example of bilinear multipliers.
If   $\mu$  is a Borel regular measure in $\R$ we denote
$\hat\mu(\xi)=\int_\R e^{-2\pi i x\xi}d\mu(x)$ its Fourier
transform.

\begin{proposition} \label{meas} Let $\Phi_1$, $\Phi_2$ and $\Phi_3$ be Young functions such that
\[ \Phi_1^{-1}(x) \Phi_2^{-1}(x) \leq \Phi_3^{-1}(x), x\in \R. \]
If  $(\alpha,\beta)\in \R^2$ and $m(\xi,\eta)=\hat{\mu}(\alpha\xi + \beta\eta)$ where $\mu$ is a regular Borel measure on $\R$ then $m\in \MBR$  and $\Vert m \Vert_{(\Phi_1,\Phi_2,\Phi_3)} \leq 2\Vert\mu \Vert_1$.
\end{proposition}

\begin{proof} Let us first rewrite the value $B_m (f,g)$ for each $f,g\in \mathcal S(\R)$ as
follows: \ba B_m(f,g)(x)&=&\int_{\R^2} \hat f(\xi)\hat g(\eta)
\hat \mu(\alpha\xi+\beta\eta)e^{2\pi i (\xi+\eta)x}d\xi d\eta\\
&=&\int_{\R^2} \hat f(\xi)\hat g(\eta) (\int_\R e^{-2\pi
i(\alpha\xi+
\beta\eta)t} d\mu(t))e^{2\pi i (\xi+\eta)x}d\xi d\eta\\
&=&\int_\R (\int_{\R^2} \hat f(\xi)\hat g(\eta)e^{2\pi i (x-\alpha
t)\xi}e^{2\pi i(x-\beta t)\eta}d\xi d\eta )d\mu(t)\\&=&\int_\R
f(x-\alpha t)g(x-\beta t)d\mu(t)\\
&=&\int_\R
\tau_{\alpha t}f(x)\tau_{\beta t} g(x)d\mu(t).\ea Hence, using Minkowski's
inequality,  (\ref{holder}) and invariance under traslations one gets
\begin{eqnarray*} N_{\Phi_3}(B_m(f,g))&\le& \int_\R N_{\Phi_3}(\tau_{\alpha t}f\tau_{\beta t} g)d|\mu|(t)\\
&\le& 2\int_\R N_{\Phi_1}(f) N_{\Phi_2}(g)d|\mu|(t)\\
&=& 2\|\mu\|_1 N_{\Phi_1}(f) N_{\Phi_2}(g).
 \end{eqnarray*}
 This gives the result.
\end{proof}

This basic example combined with the procedures exhibited in Proposition \ref{convo} produces a number of multipliers in this setting.

Also, if we consider a complementary pair of Young functions, then we can give the following result as a corollary of Proposition \ref{meas}.

\begin{corollary}
Let $(\Phi,\Psi)$ be a complementary pair of Young functions. If  $(\alpha,\beta)\in \R^2$ and $m(\xi,\eta)=\hat{\mu}(\alpha\xi + \beta\eta)$ where $\mu$ is a regular Borel measure on $\R$ then $m\in \mathcal {BM}_{(\Phi, \Psi, 1)}$  and $\Vert m \Vert_{(\Phi,\Psi,1)} \leq 4\Vert\mu \Vert_1$.
\end{corollary}
\begin{proof}
It is enough to take $\Phi_1=\Phi$, $\Phi_2=\Psi$ and $\Phi_3(x)=\frac{1}{2}|x|$, $x\in \R$ in Proposition \ref{meas}, since $\Phi$ and $\Psi$ satisfy the inequality \eqref{comp}, noticing that $L^{\Phi_3}(\R)=L^1(\R)$
and $N_{\Phi_3}(f)=\frac{1}{2}\|f\|_1$ for any $f\in L^1(\R)$.
\end{proof}

Let us now give a necessary condition for multipliers homogeneous of degree $0$. This will depend upon the Boyd indices of the spaces.
Recall that for a rearrangement invariant Banach space $X$ one defines
$$h_X(t)= \sup_{f\ne 0} \frac{\|D_{1/t}f^*\|_{\tilde X}}{\|f^*\|_{\tilde X}}, \quad t>0$$
where $\tilde X$ is the r.i. space defined on $(0,\infty)$ with the same distribution function.
The Boyd indices (see \cite[page 149]{BS}) are given by
$$\underline\alpha _X= \lim_{t\to 0} \frac{\log h_X(t)}{\log t}, \quad \overline\alpha _X= \lim_{t\to \infty} \frac{\log h_X(t)}{\log t}.$$
We denote by $\underline\alpha_\Phi$  and $\overline\alpha_\Phi$ the case $X=L^\Phi(\R)$.

\begin{proposition} Let $m\in \MBR$ a non zero multiplier such that  $m(t\xi,t\eta)= m(\xi,\eta)$ for any $t>0$. Then
\be \label{fo1}
\overline\alpha _{\Phi_3}\ge \underline\alpha _{\Phi_1}+\underline\alpha _{\Phi_2}
\ee
and
\be \label{f2}
\underline\alpha _{\Phi_3}\le \overline\alpha _{\Phi_1}+\overline\alpha _{\Phi_2}
\ee
\end{proposition}
\begin{proof} From assumption $D_t m=  m$ for $t>0$. Using now Proposition
\ref{l2} we can write $$\|m \|_{(\Phi_1,\Phi_2,\Phi_3)}\le C_{\Phi_3}({1/t}) C_{\Phi_1}(t)C_{\Phi_2}(t)\| m\|_{(\Phi_1,\Phi_2,\Phi_3)}, \quad t>0.$$
 It is elementary to show that
 $C_\Phi(t)=h_{L^\Phi}(1/t)$. Hence, denoting by  $h_\Phi=h_{L^\Phi}$, we have
 $$ h_{\Phi_3}({t}) h_{\Phi_1}(1/t)h_{\Phi_2}(1/t)\ge 1, \quad t>0.$$
Therefore
$$ \log h_{\Phi_3}(t)+ \log h_{\Phi_1}(1/t)+ \log h_{\Phi_2}(1/t)  \ge 0, \quad t>0$$
This shows that
$$\frac{\log h_{\Phi_3}(t)}{\log t} -\frac{\log h_{\Phi_1}(1/t)}{\log(1/t)} -\frac{\log h_{\Phi_2}(1/t)}{\log(1/t)}\ge 0, \quad t\ge 1$$
$$\frac{\log h_{\Phi_3}(t)}{\log t} -\frac{\log h_{\Phi_1}(1/t)}{\log(1/t)} -\frac{\log h_{\Phi_2}(1/t)}{\log(1/t)}\le 0, \quad  0<t<1.$$
Hence making limits as $t \to \infty$ and $t \to 0$ one obtains (\ref{fo1}) and (\ref{f2}) respectively.
\end{proof}

\begin{remark} Let $m\in \MBR$ and $m(t\xi,t\eta)= m(\xi,\eta)$ for any $t>0$.
In the case $\underline \alpha_{\Phi_3}= \overline\alpha_{\Phi_3}$ one has$$ \underline\alpha _{\Phi_1}+\underline\alpha _{\Phi_2}\le \alpha_{\Phi_3}\le \overline\alpha _{\Phi_1}+\overline\alpha _{\Phi_2}$$
In the case $\underline \alpha_{\Phi_i}= \overline\alpha_{\Phi_i}$ for $i=1,2$ one has,
$$ \underline\alpha_{\Phi_3}\le \alpha _{\Phi_1}+\alpha _{\Phi_2}\le \overline\alpha _{\Phi_3}.$$

For Orlicz spaces where $\underline \alpha_{\Phi_i}= \overline\alpha_{\Phi_i}$ for $i=1,2,3$ the Bilinear Hilbert transform $m(\xi,\eta)= sign (\xi-\eta)$ can only belong to $\MBR$ whenever $\alpha_{\Phi_3}=\alpha_{\Phi_1}+\alpha_{\Phi_2}$.
\end{remark}

\section{Bilinear multipliers when $m(\xi,\eta)=M(\xi-\eta)$}

Let us restrict ourselves to a class of multipliers where
$m(\xi, \eta)=M(\xi-\eta)$ for some function $M$ defined in $\R$. As in the
introduction we use the notation $\tilde\M_{(\Phi_1,\Phi_2,\Phi_3)}(\R)$ for
the space of
 locally integrable functions $M:\R\to \C$  such that
 $m(\xi,\eta)=M(\xi-\eta)\in \MBR,$
 We keep the notation $\|M\|_{(\Phi_1,\Phi_2,\Phi_3)}= \|B_m\|.$

We recall several formulations for  $B_M$ (see \cite[Proposition 3.3]{B2}): Let $M\in L^1_{loc}(\R)$, $f, g\in {\mathcal P}(\R)$.
Then \be \label{exp1} B_M(f,g)(x)=\frac{1}{2}\int_{\R^2} \hat
f(\frac{u+v}{2})\hat g(\frac{u-v}{2})
 M(v)e^{2\pi i u x}du dv\ee

 \begin{equation}\label{f1}B_M(f,g)(-x)= \int_{\R}
(\widehat{\tau_x g}*M)(\xi) \widehat{\tau_x
f}(\xi)d\xi.\end{equation}

A basic characterization for integrable symbols is the following (see \cite[Proposition 3.4]{B2}): If $M\in L^1(\R)$ and $K=\check M$, where $\check M(\xi)=\hat M(-\xi)$,  and $f,g\in \mathcal P(\R)$ then
\be \label{form1} B_M(f,g)= \int_\R f(x-t)g(x+t)K(t)dt.
\ee

A first elementary example of multiplier in $\tilde\M_{(\Phi_1,\Phi_2,\Phi_3)}(\R)$ is giving  selecting $\alpha=1$ and $\beta=-1$ in Proposition \ref{meas} obtaining the following result (which follows from (\ref{holder}):

\begin{theorem} \label{mt1}
Let $\Phi_1$, $\Phi_2$ and $\Phi_3$ be Young functions such that
\[ \Phi_1^{-1}(x) \Phi_2^{-1}(x) \leq \Phi_3^{-1}(x), x\in \R. \]

If $\mu\in M(\R)$ and $M(\xi)=\hat{\mu}(\xi)$ then  $ M\in\tilde\M_{(\Phi_1,\Phi_2,\Phi_3)}(\R)$. Moreover
$$\|M\|_{(\Phi_1,\Phi_2,\Phi_3)}\le 2\|\mu\|_1.$$
\end{theorem}

Another elementary case  is the following one.

\begin{theorem}\label{mt2}
Let $\Phi_1$, $\Phi_2$ and $\Phi_3$ be Young functions such that
\[ \Phi_1^{-1}(x) \Phi_2^{-1}(x) \leq x\Phi_3^{-1}(x), x\in \R. \]

If $M\in L^1(\R)$ then  $ M\in\tilde\M_{(\Phi_1,\Phi_2,\Phi_3)}(\R)$. Moreover
$$\|M\|_{(\Phi_1,\Phi_2,\Phi_3)}\le 2C_{\Phi_3}(2)\|M\|_1.$$
\end{theorem}
\begin{proof}
Making the change of variable $\beta=\xi-\eta$ and $\gamma= \xi$
\begin{eqnarray*}
 B_M(f,g)(x) & = & \int_{\R}\int_{\R} \hat{f}(\xi) \hat{g}(\eta) M(\xi-\eta) e^{2\pi\textit{i}(\xi+\eta)x} d\xi d\eta  \\
 & = & \int_{\R}\int_{\R} \hat{f}(\gamma) \hat{g}(\gamma-\beta) M(\beta) e^{2\pi\textit{i}(2\gamma-\beta)x} d\gamma d\beta\\
		    & = & \int_{\R} \left(\int_{\R} \widehat{(f\ast M_\beta g)}(\gamma)e^{4\pi i \gamma x}d\gamma\right) M(\beta)e^{-2\pi i \beta x}d\beta \\
	  & = & \int_{\R} \left( f\ast M_\beta g\right)(2x) M(\beta)e^{-2\pi i \beta x}d\beta
\end{eqnarray*}

Then by taking norm of this expression in $L^{\Phi_3}(\R)$, and using (\ref{young}), we obtain
\[ N_{\Phi_3}(B_M(f,g)) \leq C_{\Phi_3}(2)\int_{\R}  N_{\Phi_3}(f\ast M_\beta g) \vert M(\beta)\vert d\beta \leq  2C_{\Phi_3}(2) N_{\Phi_1}(f) N_{\Phi_2}(g) \Vert M\Vert_1 .\]
The proof is then complete.
\end{proof}

Remark that, if we consider the complementary pair of Young functions $(\Phi,\Psi)$, then we could also obtain the following new result as a corollary of Theorem \ref{mt2}.

\begin{corollary}
Let $(\Phi,\Psi)$ be a complementary pair of Young function. If $M\in L^1(\R)$ then  $ M\in\tilde\M_{(\Phi,\Psi,\infty)}(\R)$. Moreover $\|M\|_{(\Phi,\Psi,\infty)}\le 2\|M\|_1.$
\end{corollary}
\begin{proof}
We take in  Theorem \ref{mt2} the functions $\Phi_1=\Phi$, $\Phi_2=\Psi$ and $\Phi_3$ is such a way that $\Phi_3^{-1}=2$, that is to say $\Phi_3(x)=0$ for $ |x|\leq 2$ and $\Phi_3(x)=\infty$ for $|x|>2$. Then the proof is complete since $L^{\Phi_3}(\R)=L^\infty(\R)$ and the complementary pair of Young functions satisfy the inequality \eqref{comp}.
\end{proof}

As in the previous section we can  generate  new multipliers in $\mul$ using the following methods and the previous examples. The proof follows the same ideas as in \cite{B2} and Proposition \ref{convo} and it is left to the reader.

\begin{proposition}  \label{propi2} Let $\phi\in L^1(\R)$ and $M\in
\mul$. Then

\begin{enumerate}[\sf \ \ \ (a)]
\item
$\phi*M\in \mul$ and $\|\phi*M\|_{p_1,p_2,p_3}\le \|\phi\|_1
\|M\|_{(\Phi_1,\Phi_2,\Phi_3)}.$
\item   $\hat\phi M\in \mul$ and
 $\|\hat\phi M\|_{(\Phi_1,\Phi_2,\Phi_3)}\le \|\phi\|_1
\|M\|_{(\Phi_1,\Phi_2,\Phi_3)}.$

\item If $\psi\in L^1(\R^+, W)$  then $M_\psi(\xi)=\int_0^\infty M(t\xi)\psi(t)dt \in \mul$.
Moreover $\|M_\psi\|_{(\Phi_1,\Phi_2,\Phi_3)}\le \|\psi\|_{L^1(\R^+, W)}
\|M\|_{(\Phi_1,\Phi_2,\Phi_3)}.$
\end{enumerate}

\end{proposition}

\section{On necessary conditions for $\mul\ne \{0\}$}
Let  us show that the classes $\mul$ are reduced to $\{0\}$ in certain cases. We shall use arguments from \cite[Theorem 3.7, Theorem 3.9]{B2} and  \cite[Theorem 5.10]{R}.

We need the following lemma to give a result about the bilinear multipliers in the class $M_{(\Phi_1,\Phi_2,\Phi_3)}$.

\begin{lemma} \label{lemma normPhi geq norm1 }
Let $g$ be a continuous function in $\R$ with $supp(g)\subset [0,a]$ for some $a>0$ and let $\Phi$ be a Young function. Then
\[ N_\Phi(\Sigma_{k=0}^N \epsilon_k \tau_{[a+1]k}g) \geq \frac{1}{a\Phi^{-1}(\frac{1}{a(N+1)})}\Vert g\Vert_1 \]
where $\epsilon_k \in \{\pm 1\}$.
\end{lemma}

\begin{proof}
Note that $supp(\tau_yg) \subset [y,y+a]$ and then $\tau_{[a+1]k}g$ are disjointly supported. Hence if $h = \Sigma_{k=0}^N \epsilon_k \tau_{[a+1]k}g$, $I_0=[0,a]$ and $I_k=[[a+1]k,[a+1]k+a]$ then, using Jensen's inequality, one has

\begin{eqnarray*}
N_\Phi(h) &=& \inf\{ \lambda>0 : \int_\R \Phi(\frac{\vert h(x)\vert}{\lambda})dx \leq 1\}  \\
 &=& \inf\{\lambda>0: \Sigma_{k=0}^N \int_{I_k} \Phi(\frac{\vert \tau_{[a+1]k}g(x)\vert}{\lambda})dx\leq 1\}\\
 &=& \inf\{\lambda>0: \frac{1}{|I_0|} \int_{I_0} \Phi(\frac{\vert g(x)\vert}{\lambda})dx\leq \frac{1}{a(N+1)}\}\\
 &\geq & \inf\{\lambda>0: \Phi(\frac{1}{a} \int_{I_0} \frac{\vert g(x)\vert}{\lambda})dx\leq \frac{1}{a(N+1)}\}\\
&=& \frac{1}{a\Phi^{-1}(\frac{1}{a(N+1)})} \Vert g\Vert_1
\end{eqnarray*}
where the last equality follows same  argument as in the proof of Lemma \ref{l1}.
\end{proof}

\begin{theorem} \label{necessary cond1}
Let $\Phi_1,\Phi_2, \Phi_3$ be Young functions.

(i) If
$$\sup_{x\in \R} \frac{\Phi_1^{-1}\left( x\right) \Phi_2^{-1}\left( x\right)} {\Phi_3^{-1}\left( x\right)}<\infty$$
then $\mul\ne \{0\}$.

(ii) If  $\tilde{\mathcal M}_{(\Phi_1, \Phi_2, \Phi_3)}(\R) \ne \{0\}$ then for all $a>0$ one has
\[ \sup_{N\geq 1} \frac{\Phi_1^{-1}\left( \frac{1}{Na}\right) \Phi_2^{-1}\left( \frac{1}{Na}\right)} {\Phi_3^{-1}\left( \frac{1}{Na}\right)} <\infty. \]
\end{theorem}

\begin{proof}
(i) follows from Theorem \ref{mt1}.

(ii) Let $0\ne M \in \tilde{M}_{(\Phi_1,\Phi_2,\Phi_3)}(\R)$. Using Proposition \ref{propi2}  we may assume that there exists $0\ne M \in L^1(\R)\cap \tilde{M}_{(\Phi_1,\Phi_2,\Phi_3)}(\R)$. Hence, from (\ref{form1}) one has that
\[ B_M(f,g)(x) = \int_{[x-a,x]\cap [-x,-x+a]} f(x-t) g(x+t) \hat{M}(-t) dt \]
for any $f$ and $g$ continuous functions compactly supported in $[0,a]$.
Consider the Rademacher system in $[0,1]$ and observe that for each $N$ and $y \in \R$, the orthonormality of the system gives
\[ \int_{0}^1 B_M \left( \Sigma_{k=0}^N r_k(t)\tau_{ky}f , \Sigma_{k=0}^N r_k(t)\tau_{ky}g  \right) dt = \Sigma_{k=0}^N B_M(\tau_{ky}f, \tau_{ky}g)  \]
Therefore, since $B_M(\tau_{ky}f, \tau_{ky}g) = \tau_{ky} B_M(f,g)$, we have
\begin{equation} \label{radem}
 \int_{0}^1 B_M \left( \Sigma_{k=0}^N r_k(t)\tau_{ky}f , \Sigma_{k=0}^N r_k(t)\tau_{ky}g  \right) dt = \Sigma_{k=0}^N \tau_{ky} B_M(f,g)
\end{equation}
for any $f,g$ compactly supported in $[0,a]$.
Now, let us consider the functions $f=g=\chi_{[0,a]}$, where $a>0$ is arbitrary constant, $y=[a+1]$ where $[\cdot]$ is the integer part. For each $N\in \N$ and $t\in [0,1]$ we denote
\[  f^t_N(x) = \Sigma_{k=1}^N r_k(t)\chi_{[[a+1]k,[a+1]k+a]}(x) .\]

Then for the functions $f$ and $f_N$, by using \eqref{radem} we have
\[ \int_{0}^1 B_M \left( f^t_N, f^t_N \right) dt =  \Sigma_{k=0}^N \tau_{[a+1]k} B_M(f,f) \]
where $supp B_M(f,f) \subset [0,2a]$.

By taking norm of the right hand side of this equality in $L^{\Phi_3}(\R)$ and using the Lemma \ref{lemma normPhi geq norm1 } we observe that

\begin{equation}\label{Rademacher_norm1}
N_{\Phi_3}\left(\Sigma_{k=0}^N \tau_{[a+1]k} B_M(f,f)\right) \geq \frac{\Vert B_M(\chi_{[0,a]},\chi_{[0,a]}) \Vert_1}{a\Phi_3^{-1}(\frac{1}{(N+1)a})}
\end{equation}

On the other hand, by using Minkowski's inequality and Lemma \ref{l1} we have

\begin{eqnarray*}\label{Rademacher_norm2}
 N_{\Phi_3}(\int_{0}^1 B_M \left( f^t_N, f^t_N \right) dt )
   &\leq&  \int_{0}^1 N_{\Phi_3}(B_M(f^t_N,f^t_N)) dt  \\
   &\leq&  \int_{0}^1  \Vert B_M\Vert N_{\Phi_1}(f^t_N) N_{\Phi_2}(f^t_N)) dt  \\
   &=&   \Vert B_M\Vert \frac{1}{\Phi_1^{-1}(\frac{1}{Na})} \frac{1}{\Phi_2^{-1}(\frac{1}{Na})}
\end{eqnarray*}
which  combining with \eqref{Rademacher_norm1} gives, for each $a>0$ and for all $N\in \N$,
$$
\frac{1}{a\Phi_3^{-1}(\frac{1}{Na})} \Vert B_M(f,f) \Vert_1 \leq    \Vert B_M\Vert \frac{1}{\Phi_1^{-1}(\frac{1}{Na})} \frac{1}{\Phi_2^{-1}(\frac{1}{Na})}
.$$
 This implies that for any $a>0$ there exists $C = C_a > 0$ such that
\[
\sup_N\frac{\Phi_1^{-1}(\frac{1}{Na})  \Phi_2^{-1}(\frac{1}{Na})}{\Phi_3^{-1}(\frac{1}{Na})} \leq C_a
\]
where
$
C_a = \frac{a \Vert B_m\Vert}{\Vert B_M(\chi_{[0,a]},\chi_{[0,a]})\Vert_1} > 0.
$
This completes the proof.
\end{proof}

Note that, if we take $\Phi_i(x)=\vert x\vert^{p_i}$ for $i=1,2,3$, then $L^{\Phi_i}(\R)=L^{p_i}(\R)$ and $\Phi_i^{-1}(x)=\vert x\vert^{1/p_i}$. Theorem \ref{necessary cond1} becomes now $sup_{N\geq 1}(\frac{1}{Na})^{\frac{1}{p_1}+\frac{1}{p_2}-\frac{1}{p_3}}<\infty$. This gives the following corollary.

\begin{corollary}[\cite{B2}]
Let $p_1,p_2, p_3\geq 1$ such that $\frac{1}{p_1}+\frac{1}{p_2} < \frac{1}{p_3}$. Then $\tilde{M}_{(p_1 , p_2,p_3)}(\R)=\{0\}$.
\end{corollary}

Let us now use another approach following \cite{B2}  to get other necessary
conditions on multipliers.

\begin{lemma}\label{Lemma nonzero multiplier}
Let $M \in \tilde{M}_{(\Phi_1 , \Phi_2 , \Phi_3)}(\R)$ such that $F_M(\lambda)=\vert \int_\R e^{-\lambda^2 v^2} M(v) dv\vert <\infty$ for all $\lambda>0$. Then there exists a constant $A>0$ such that
\begin{equation}
A\lambda F_M(\lambda)  \leq   C_{\Phi_1}(1/\lambda)  C_{\Phi_2}(1/\lambda) C_{\Phi_3}(\lambda), \quad \lambda>0.
\end{equation}
\end{lemma}
\begin{proof}
Let $\lambda>0$ and recall that $G_\lambda(x)=\frac{1}{\lambda}G(\frac{x}{\lambda})$ with $\widehat {G_\lambda}= D_\lambda \hat G$. Take $G$  such that $\hat{G}(\xi)=e^{-2\xi^2}$. Using  formula (\ref{exp1}) one has

	\begin{eqnarray*}
B_M(G_\lambda, G_\lambda)(x)
&=& \frac{1}{2} \int_\R \int_\R \hat{G}(\frac{u+v}{2}\lambda) \hat{G}(\frac{u-v}{2}\lambda) M(v) e^{2\pi iux} du dv  \\
&=& \frac{1}{2} (\int_\R e^{-\lambda^2 u^2} e^{2\pi iux} du) (\int_\R e^{-\lambda^2 v^2} M(v)dv)  \\
&=& \frac{1}{2} G_{\lambda/\sqrt 2}(x) (\int_\R e^{-\lambda^2 v^2} M(v)dv)
	\end{eqnarray*}

Since $M \in \tilde{M}_{(\Phi_1 , \Phi_2 , \Phi_3 )}(\R)$ we have

	\begin{equation}
\frac{1}{2} N_{\Phi_3}(G_{\lambda/\sqrt 2}) F_M(\lambda) \leq  \Vert M\Vert_{(\Phi_1 , \Phi_2 , \Phi_3 )} N_{\Phi_1}(G_\lambda ) N_{\Phi_2}(G_\lambda).
	\end{equation}

Since $$ N_{\Phi_i}(G_\lambda)\le \frac{C_{\Phi_i}(1/\lambda)}{\lambda} N_\Phi(G), \quad i=1,2$$
 and, using $G= D_{\lambda/\sqrt 2}D_{\sqrt 2/\lambda} G$, also  $$\frac{ \sqrt 2N_{\Phi_3}(G)}{\lambda C_{\Phi_3}(1/\sqrt 2) C_{\Phi_3}(\lambda)}\le \frac{ \sqrt 2N_{\Phi_3}(G)}{\lambda C_{\Phi_3}(\lambda/\sqrt 2)}\le N_{\Phi_3}(G_{\lambda/\sqrt 2})$$

We can write
	$$
 \frac{ \sqrt 2N_{\Phi_3}(G)}{2\lambda C_{\Phi_3}(1/\sqrt 2) C_{\Phi_3}(\lambda)} F_M(\lambda)  \leq  \frac{1}{\lambda^2} \Vert M\Vert_{(\Phi_1 , \Phi_2 , \Phi_3 )} C_{\Phi_1}(1/\lambda) N_{\Phi_1}(G) C_{\Phi_2}(1/\lambda) N_{\Phi_2}(G).
	$$

Hence we have
\begin{equation}
A\lambda F_M(\lambda)  \leq   C_{\Phi_1}(1/\lambda)  C_{\Phi_2}(1/\lambda) C_{\Phi_3}(\lambda)
	\end{equation}
for some constant $A>0$.
\end{proof}

\begin{theorem}
If there exists a non-zero continuous and integrable function $M\in \hat{M}_{(\Phi_1 , \Phi_2 , \Phi_3)}(\R)$ then
	\begin{equation} \label{final1}
 \liminf_{\lambda\to 0}  C_{\Phi_1}(\lambda)  C_{\Phi_2}(\lambda) C_{\Phi_3}(\frac{1}{\lambda}) > 0
	\end{equation}
and

	\begin{equation}\label{final2}
 \liminf_{\lambda\to \infty} \lambda C_{\Phi_1}(\lambda)  C_{\Phi_2}(\lambda) C_{\Phi_3}(\frac{1}{\lambda}) > 0
	\end{equation}
\end{theorem}

\begin{proof}
Let $y\in \R$ such that $M(y)\ne 0$. By using Lemma \ref{Lemma nonzero multiplier} to the function $M(y-\cdot) $ we obtain
	\begin{equation*}
A\lambda\vert  \int_\R e^{-\lambda^2 \xi^2} M(y-\xi)d\xi \vert \leq    C_{\Phi_1}(1/\lambda)  C_{\Phi_2}(1/\lambda) C_{\Phi_3}(\lambda).
	\end{equation*}

Therefore, using that $M\in C_0(\R)$, the convolution with approximation of the identity  and taking limits as $\lambda \to \infty$ one gets

	\begin{equation*}
\lim_{\lambda \to \infty} \vert \lambda \int_\R e^{-\lambda^2 \xi^2} M(y-\xi)d\xi \vert = \sqrt \pi\vert M(y)\vert >0.
	\end{equation*}
This gives  (\ref{final1}).

 Since $\hat M\ne 0$ there exists $y\in \R$ such that $\hat M(y)\ne 0$. Using again Lemma \ref{Lemma nonzero multiplier}, applied to $M_{-y}M$ we obtain

	\begin{equation*}
A\vert \int_\R e^{-\lambda^2 \xi^2} e^{-2\pi i \xi y} M(\xi)d\xi \vert \leq \frac{1}{\lambda} C_{\Phi_1}(1/\lambda)  C_{\Phi_2}(1/\lambda) C_{\Phi_3}(\lambda).
	\end{equation*}

Therefore, taking limits as  $\lambda \to 0$ we get

	\begin{equation*}
\lim_{\lambda \to 0} \vert \int_\R e^{-\lambda^2 \xi^2} e^{-2\pi i \xi y} M(\xi)d\xi \vert = \vert \hat{M}(y)\vert >0.
	\end{equation*}
Hence we get (\ref{final2}).
\end{proof}

\begin{corollary}\label{if alpha=0 or beta=0 then M=0}
Let $\Phi_1 , \Phi_2 , \Phi_3$ be Young functions and let
	\begin{equation*}
\alpha = \liminf_{\lambda\to 0}  C_{\Phi_1}(\lambda)  C_{\Phi_2}(\lambda) C_{\Phi_3}(\frac{1}{\lambda})
	\end{equation*}
and

	\begin{equation*}
\beta = \liminf_{\lambda\to \infty}  \lambda C_{\Phi_1}(\lambda)  C_{\Phi_2}(\lambda) C_{\Phi_3}(\frac{1}{\lambda})
	\end{equation*}

If $\alpha=0$ or $\beta=0$ then $\tilde{M}_{(\Phi_1 , \Phi_2 , \Phi_3)}(\R) = \{0\}$.
\end{corollary}

\begin{corollary} (see \cite{B2, V}) Let $1\le p_i<\infty$ for $i=1,2,3$. If $\tilde{\mathcal M}_{(p_1 , p_2 , p_3)}(\R) \ne \{0\}$ then
$$ \frac{1}{p_3}\le \frac{1}{p_1}+\frac{1}{p_2}\le \frac{1}{p_3}+1.$$
\end{corollary}
\begin{proof} For $\Phi_i(x)=\vert x\vert^{p_i}$ for $i=1,2,3$ the dilation operator $D_\lambda$ has norm $C_{p_i}(\lambda)=\lambda^{-1/p_i}$ for $i=1,2,3$.
 In this case the constants $\alpha$ and $\beta$ in the Corollary \ref{if alpha=0 or beta=0 then M=0} become
\begin{equation}
\alpha = \liminf_{\lambda\to 0} \lambda^{\frac{1}{p_3}-\frac{1}{p_1}-\frac{1}{p_2}}
\end{equation}
and
\begin{equation}
\beta = \liminf_{\lambda\to \infty} \lambda^{1+\frac{1}{p_3}-\frac{1}{p_1}-\frac{1}{p_2}}.
\end{equation} Hence $\alpha=0$ and $\beta=0$ correspond to  $\frac{1}{p_3}>\frac{1}{p_1}+\frac{1}{p_2}$ and  $1+\frac{1}{p_3}<\frac{1}{p_1}+\frac{1}{p_2}$ respectively. The result now follows from Corollary \ref{if alpha=0 or beta=0 then M=0}.
\end{proof}

\begin{remark}
The reader is also referred to the work of S. Rodriguez \cite{R} where the existence of a non-zero bilinear multiplier on r.i Banach spaces (in particular to Orlicz spaces) is  related to Boyd indices of the spaces.
\end{remark}

{\it Acknowledgment:}
The second named author would like to express his gratitude to
the Department of Mathematical Analysis of Valencia University for their hospitality
during his stay in Valencia as a visiting researcher and to T{\"U}B{\'I}TAK for their support
to make this work. We both would like to thank M.J Carro for calling to our attention  \cite[Theorem 5.10]{R}.

\end{document}